\documentclass[a4paper,11pt]{article}
\usepackage{amsmath}
\usepackage{amsfonts}
\usepackage{supertabular}
\usepackage[latin9]{inputenc}
\usepackage[english]{varioref}
\usepackage{dcolumn}
\usepackage[height=22cm , width = 16cm , top = 4cm , left = 3cm, a4paper]{geometry}
\usepackage[a4paper]{geometry}
\usepackage[final]{graphicx}
\usepackage{epsfig}
\usepackage{pstricks}
\usepackage{psfrag}
\usepackage{rotating}
\usepackage{supertabular}
\usepackage{booktabs}
\usepackage{delarray}
\usepackage{rotating}
\usepackage[numbers]{natbib}
\usepackage{subfigure}
\usepackage{nextpage}
\usepackage{layout}
\usepackage{amsthm}
\usepackage{dsfont}
\usepackage{hyperref}
\usepackage[hypcap]{caption}
\usepackage[autostyle]{csquotes}
\usepackage{color}
\usepackage{calrsfs}

\numberwithin{equation}{section}
{                     
{                     
{                       
{                       

\theoremstyle{plain}
\newtheorem{thm}{Theorem}[section]

\newtheorem{lem}[thm]{Lemma}

\newtheorem{definition}[thm]{Definition}

\newtheorem{hypp}[thm]{Hypotheses}
\parindent0cm
\newcommand{\enter}{\bigskip}

\date{}

\begin{document}
 \author{Prasanta Kumar Barik$^1$, {Ankik Kumar Giri$^1$}\footnote{Corresponding author. Tel +91-1332-284818 (O);  Fax: +91-1332-273560  \newline{\it{${}$ \hspace{.3cm} Email address: }}ankikgiri.fma@iitr.ac.in/ankik.math@gmail.com} and Philippe Lauren\c{c}ot$^2$\\
\footnotesize  ${}^1$Department of Mathematics, Indian Institute of Technology Roorkee,\\ \small{ Roorkee-247667, Uttarakhand, India}\\
\footnotesize  ${}^2$Institut de Mathématiques de Toulouse
 UMR 5219, Université de Toulouse, CNRS \\ \small{ F-31062 Toulouse Cedex 9, France}
  }

\title {Mass-conserving solutions to the Smoluchowski coagulation equation with singular kernel}

\maketitle


\begin{quote}
{\small {\em\bf Abstract.} Global weak solutions to the continuous Smoluchowski coagulation equation (SCE) are constructed for coagulation kernels featuring an algebraic singularity for small volumes and growing linearly for large volumes, thereby extending previous results obtained in Norris (1999) and Cueto Camejo \& Warnecke (2015). In particular, linear growth at infinity of the coagulation kernel is included and the initial condition may have an infinite second moment. Furthermore, all weak solutions (in a suitable sense) including the ones constructed herein are shown to be mass-conserving, a property which was proved in Norris (1999) under stronger assumptions. The existence proof relies on a weak compactness method in $L^1$ and a by-product of the analysis is that both conservative and non-conservative approximations to the SCE lead to weak solutions which are then mass-conserving.\enter
}
\end{quote}


\section{Introduction}\label{existintroduction1}
The kinetic process in which particles undergo changes in their physical properties is called a particulate process. The study of particulate processes is a well-known subject in various branches of engineering, astrophysics, physics, chemistry and in many other related areas. During the particulate process, particles merge to form larger particles or break up into smaller particles. Due to this process, particles change their size, shape and volume, to name but a few. There are various types of particulate processes such as coagulation, fragmentation, nucleation and growth for instance. In particular, this article mainly deals with the coagulation process which is governed by the Smoluchowski coagulation equation (SCE). In this process, two particles coalesce to form a larger particle at a particular instant.

The SCE is a nonlinear integral equation which describes the dynamics of evolution of the concentration $g(\zeta,t)$ of particles of volume $\zeta>0$ at time $t \geq 0$ \cite{Smoluchowski:1917}. The evolution of $g$ is given by
\begin{equation}\label{sce}
\frac{\partial g(\zeta,t)}{\partial t}  = \mathcal{B}_c(g)({\zeta}, t) -\mathcal{D}_c(g)({\zeta}, t),\ \qquad (\zeta,t)\in (0,\infty)^2,\
\end{equation}
with initial condition
\begin{equation}\label{1in1}
g(\zeta,0) = g^{in}(\zeta)\geq 0,\ \qquad \zeta\in (0,\infty),\
\end{equation}
where the operator $\mathcal{B}_c$ and $\mathcal{D}_c$ are expressed as
\begin{equation}\label{Birthterm}
 \mathcal{B}_c(g)({\zeta}, t):=\frac{1}{2} \int_{0}^{\zeta} \Psi (\zeta -\eta,\eta)g(\zeta -\eta,t)g(\eta,t)d\eta
\end{equation}
and
\begin{equation}\label{Deathterm}
\mathcal{D}_c(g)({\zeta}, t):=\int_{0}^{\infty} \Psi(\zeta,\eta)g(\zeta,t)g(\eta,t)d\eta.
\end{equation}
 Here $\frac{\partial g(\zeta,t)}{\partial t}$ represents the time partial derivative of the concentration of  particles of volume $\zeta$ at time $t$. In addition, the non-negative quantity $\Psi(\zeta, \eta)$ denotes the interaction rate at which particles of volume $\zeta$ and particles of volume $\eta$ coalesce to form larger particles. This rate is also known as the coagulation kernel or coagulation coefficient. The first and last terms $\mathcal{B}_c(g)$ and  $\mathcal{D}_c(g)$ on the right-hand side to (\ref{sce}) represent the formation and disappearance of particles of volume $\zeta$ due to coagulation events, respectively.

Let us define the total mass (volume) of the system at time $t\ge 0$ as:
\begin{equation}\label{Totalmass}
 \mathcal{M}_1(g)(t):=\int_0^{\infty}{\zeta} g(\zeta,t)d{\zeta}.
\end{equation}

According to the conservation of matter, it is well known that the total mass (volume) of particles is neither created nor destroyed. Therefore, it is expected that the total mass (volume) of the system remains conserved throughout the time evolution prescribed by \eqref{sce}--\eqref{1in1}, that is, $\mathcal{M}_1(g)(t)=\mathcal{M}_1(g^{in})$ for all $t\ge 0$. However, it is worth to mention that, for the multiplicative coagulation kernel $\Psi(\zeta, \eta)=\zeta \eta$, the total mass conservation fails for the SCE at finite time $t=1$, see \cite{Leyvraz:1981}. The physical interpretation is that the lost mass corresponds to \enquote{particles of infinite volume} created by a runaway growth in the system due to the very high rate of coalescence of very large particles. These particles, also referred to as \enquote{giant particles} \cite{Aldous:1999} are interpreted in the physics literature as a different macroscopic phase, called a \emph{gel}, and its occurrence is called the \emph{sol-gel transition} or \emph{gelation transition}. The earliest time $T_g \geq 0$ after which mass conservation no longer holds is called the \emph{gelling time} or \emph{gelation time}.

Since the works by Ball \& Carr \cite{Ball:1990} and Stewart \cite{Stewart:1989}, several articles have been devoted to the existence and uniqueness of solutions to the SCE for coagulation kernels which are bounded for small volumes and unbounded for large volumes, as well as to the mass conservation and gelation phenomenon, see \cite{Dubovskii:1996, Escobedo:2002, Escobedo:2003, Giri:2012, Laurencot:2002L, Norris:1999, Stewart:1990}, see also the survey papers \cite{Aldous:1999, Laurencot:2015, Mischler:2004} and the references therein. However, to the best of our knowledge, there are fewer  articles in which existence and uniqueness of solutions to the SCE with singular coagulation rates have been studied, see \cite{Camejo:2015I, Camejo:2015II, Escobedo:2005, Escobedo:2006, Norris:1999}. In \cite{Norris:1999}, Norris  investigates the existence and uniqueness of solutions to the SCE locally in time  when the coagulation kernel satisfies
\begin{equation}
\Psi({\zeta}, {\eta}) \leq \phi({\zeta}) \phi({\eta}),\ \qquad (\zeta,\eta)\in (0,\infty)^2,\ \label{CK1}
\end{equation}
for some sublinear function $\phi: (0, \infty) \rightarrow [0, \infty)$, {that is, $\phi$ enjoys the property $\phi(a{\zeta})\leq a\phi({\zeta})$ for all $\zeta \in (0, \infty)$ and $a \geq 1$, and the initial condition $g^{in}$ belongs to $L^1((0,\infty);\phi(\zeta)^2d\zeta)$. Mass-conservation is also shown as soon as there is $\varepsilon>0$ such that $\phi(\zeta)\ge \varepsilon\zeta$ for all $\zeta\in (0,\infty)$. In \cite{Escobedo:2006, Escobedo:2005}, global existence, uniqueness, and mass-conservation are established for coagulation rates of the form $\Psi({\zeta}, {\eta}) = {\zeta}^{\mu_1}{\eta}^{\mu_2}+{\zeta}^{\mu_2}{\eta}^{\mu_1}$ with $-1\leq \mu_1 \leq \mu_2 \leq 1$, $\mu_1+\mu_2 \in[0, 2]$, and $(\mu_1, \mu_2)\neq (0, 1)$.} Recently, global existence of weak solutions to the SCE for coagulation kernels satisfying
\begin{equation*}
\Psi({\zeta}, {\eta}) \leq k^{*}(1+{\zeta}+{\eta})^{\lambda}({\zeta}{\eta})^{-\sigma},\ \qquad (\zeta,\eta)\in (0,\infty)^2,\
\end{equation*}
with $\sigma \in [0, 1/2]$, $\lambda -\sigma \in [0,1)$, and $k^{*}>0$, is obtained in \cite{Camejo:2015I} and further extended in \cite{Camejo:2015II} to the broader class of coagulation kernels
\begin{equation}
\Psi({\zeta}, {\eta}) \leq k^{*} (1+{\zeta})^{\lambda}(1+{\eta})^{\lambda}({\zeta}{\eta})^{-\sigma},\ \qquad (\zeta,\eta)\in (0,\infty)^2,\ \label{CK2}
\end{equation}
with $\sigma \geq 0$, $\lambda -\sigma \in [0,1)$,  and $k^{*}>0$. In \cite{Camejo:2015II}, multiple fragmentation is also included and uniqueness is shown for the following restricted class of coagulation kernels
\begin{equation*}
\Psi_2({\zeta}, {\eta}) \leq k^{*}({\zeta}^{-\sigma}+{\zeta}^{\lambda-\sigma}) ({\eta}^{-\sigma}+{\eta}^{\lambda-\sigma}),\ \qquad (\zeta,\eta)\in (0,\infty)^2,\
\end{equation*}
where $\sigma \geq 0$ and $\lambda -\sigma \in [0, 1/2]$.

The main aim of this article is to extend and complete the previous results in two directions. We actually consider coagulation kernels satisfying the growth condition \eqref{CK1} for the non-negative function
\begin{equation*}
\phi_\beta(\zeta) := \max\left\{ \zeta^{-\beta}, \zeta \right\},\ \qquad \zeta\in (0,\infty),\
\end{equation*}
and prove the existence of a global mass-conserving solution of the SCE (\ref{sce})--(\ref{1in1}) with initial conditions in $L^1((0,\infty);(\zeta^{-2\beta}+\zeta)d\zeta)$, thereby removing the finiteness of the second moment required to apply the existence result of \cite{Norris:1999} and relaxing the assumption $\lambda<\sigma+1$ used in \cite{Camejo:2015II} for coagulation kernels satisfying \eqref{CK2}. Besides this, we show that any weak solution in the sense of Definition~\ref{definition} below is mass-conserving, a feature which was enjoyed by the solution constructed in \cite{Norris:1999} but not investigated in \cite{Camejo:2015I, Camejo:2015II}. An important consequence of this property is that it gives some flexibility in the choice of the method to construct a weak solution to the SCE (\ref{sce})--(\ref{1in1}) since it will be mass-conserving whatever the approach. Recall that there are two different approximations of the SCE \eqref{sce} by truncation have been employed in recent years, the so-called conservative and non-conservative approximations, see \eqref{Nonconservativedeath} below. While it is expected and actually verified in several papers that the conservative approximation leads to a mass-conserving solution to the SCE, a similar conclusion is not awaited when using the non-conservative approximation which has rather been designed to study the gelation phenomenon, in particular from a numerical point of view \cite{Filbet:2004II, Bourgade:2008}. Still, it is by now known that, for the SCE with locally bounded coagulation kernels growing at most linearly at infinity, the non-conservative approximation also allows one to construct mass-conserving solutions \cite{Filbet:2004I, Barik:2017}. The last outcome of our analysis is that, in our case, the conservative and non-conservative approximations can be handled simultaneously and both lead to a weak solution to the SCE which might not be the same due to the lack of a general uniqueness result but is mass-conserving.

We now outline the results of the paper: In the next section, we state precisely  our hypotheses on coagulation kernel and on the initial data together with the definition of solutions and the main result. In Section 3, all weak solutions are shown to be mass-conserving.  Finally, in the last section, the existence of a weak solution to the SCE \eqref{sce}--\eqref{1in1} is obtained by using a weak $L^1$ compactness method applied to either the non-conservative or the conservative approximations of the SCE.

 \section{Main result}
We assume that the coagulation kernel $\Psi$ satisfies the following hypotheses.
\begin{hypp}\label{hyppmcs}
(H1) $\Psi$ is a non-negative measurable function on $(0,\infty) \times (0,\infty)$,
\\
(H2) There are $\beta>0$ and $k>0$ such that
\begin{equation*}
\begin{array}{ll}
0 \le \Psi(\zeta,\eta) = \Psi(\eta,\zeta) \le k (\zeta\eta)^{-\beta},\ & (\zeta,\eta)\in (0,1)^2, \\
0 \le \Psi(\zeta,\eta) = \Psi(\eta,\zeta) \le k \eta \zeta^{-\beta},\ & (\zeta,\eta)\in (0,1)\times (1,\infty), \\
0 \le \Psi(\zeta,\eta) = \Psi(\eta,\zeta) \le k (\zeta+\eta),\ & (\zeta,\eta)\in (1,\infty)^2.
\end{array}
\end{equation*}
Observe that (H2) implies that
\begin{equation*}
\Psi(\zeta,\eta) \le k \max\left\{ \zeta^{-\beta} , \zeta \right\} \max\left\{ \eta^{-\beta} , \eta \right\},\ \qquad (\zeta,\eta)\in (0,\infty)^2.
\end{equation*}
\end{hypp}

Let us now mention the following interesting singular coagulation kernels satisfying hypotheses~\ref{hyppmcs}.

\begin{itemize}
\item[(a)] Smoluchowski's coagulation kernel \cite{Smoluchowski:1917} (with $\beta=1/3$)
\begin{equation*}
\Psi(\zeta,\eta)= \left( \zeta^{1/3} + \eta^{1/3} \right) \left( \zeta^{-1/3} + \eta^{-1/3} \right),\ \qquad (\zeta,\eta)\in (0,\infty)^2.
\end{equation*}
\item[(b)] Granulation kernel \cite{Kapur:1972}
\begin{equation*}
\Psi(\zeta, \eta) = \frac{(\zeta +\eta)^{\theta_1}}{({\zeta}{\eta})^{\theta_2}},~~ \mbox{where}~~ {\theta_1} \leq 1 \ \text{and}\ {\theta_2} \geq 0.
\end{equation*}
\item[(c)] Stochastic stirred froths \cite{Clark:1999}
\begin{equation*}
\Psi(\zeta, \eta) = (\zeta\eta)^{-\beta},~~ \mbox{where}~~ \beta > 0.
\end{equation*}
\end{itemize}

Before providing the statement of Theorem~\ref{thm1}, we recall the following definition of weak solutions to the SCE (\ref{sce})--(\ref{1in1}). We set $L^1_{-2\beta, 1}(0, \infty):=L^1((0, \infty ) ; ({\zeta}^{-2\beta}+\zeta)d{\zeta})$.

\begin{definition}\label{definition}
Let $T\in (0,\infty]$ and $g^{in}\in L_{-2\beta,1}^1(0,\infty)$, $g^{in}\ge 0$ a.e. in $(0,\infty)$. A non-negative real valued function $g=g({\zeta},t)$ is a weak solution to equations (\ref{sce})--(\ref{1in1}) on $[0,T)$ if $g \in \mathcal{C}({[0,T)};L^1(0, \infty))\bigcap L^{\infty}(0,T;L^1_{-2\beta, 1}(0, \infty))$
and satisfies
\begin{align}\label{wsce}
\int_0^{\infty}[ g({\zeta},t) - g^{in}({\zeta})]&\omega({\zeta})d{\zeta}=\frac{1}{2}\int_0^t \int_0^{\infty} \int_{0}^{\infty}\tilde{\omega}({\zeta},{\eta}) \Psi({\zeta},{\eta})g({\zeta},s)g({\eta},s)d{\eta}d{\zeta}ds,
\end{align}
for every $t\in (0,T)$ and $\omega \in L^{\infty}(0, \infty)$, where
\begin{equation*}
\tilde{\omega} ({\zeta},{\eta}):=\omega({\zeta}+{\eta})-\omega ({\zeta})-\omega({\eta}),  \qquad (\zeta,\eta)\in (0,\infty)^2.
\end{equation*}
\end{definition}

Now, we are in a position to state the main theorem of this paper.
\begin{thm}\label{thm1}
Assume that the coagulation kernel satisfies \textbf{hypotheses}~$(H1)$--$(H2)$ and consider a non-negative initial condition $g^{in}\in L^1_{-2\beta, 1}(0, \infty)$. There exists at least one mass-conserving weak solution $g$ to the SCE (\ref{sce})--(\ref{1in1}) on $[0,\infty)$, that is, $g$ is a weak solution to (\ref{sce})--(\ref{1in1}) in the sense of Definition~\ref{definition} satisfying $\mathcal{M}_1(g)(t) = \mathcal{M}_1(g^{in})$ for all $t\ge 0$, the total mass $\mathcal{M}_1(g)$ being defined in \eqref{Totalmass}.
\end{thm}

\section{Weak solutions are mass-conserving}

In this section, we establish that any weak solution $g$ to (\ref{sce})--(\ref{1in1}) on $[0,T)$, $T\in (0,\infty]$, in the sense of Definition~\ref{definition} is mass-conserving, that is, satisfies
\begin{equation}
\mathcal{M}_1(g)(t) = \mathcal{M}_1(g^{in}),\ \qquad t\ge 0. \label{PhL0}
\end{equation}
To this end, we adapt an argument designed in \cite[Section~3]{Ball:1990} to investigate the same issue for the discrete coagulation-fragmentation equations and show that the behaviour of $g$ for small volumes required in Definition~\ref{definition} allows us to control the possible singularity of $\Psi$.

\begin{thm}\label{themcs}
Suppose that $(H1)$--${(H2)}$ hold. Let $g$ be a weak solution to (\ref{sce})--(\ref{1in1}) on $[0, T)$ for some $T\in (0,\infty]$. Then $g$ satisfies the mass-conserving property \eqref{PhL0} for all $t\in (0,T)$.
\end{thm}

In order to prove Theorem~\ref{themcs}, we need the following sequence of lemmas.

\begin{lem}\label{lemma1}
Assume that $(H1)$--{$(H2)$} hold. Let $g$ be a weak solution to (\ref{sce})--(\ref{1in1}) on $[0, T)$. Then, for $q\in (0,\infty)$ and $t\in (0,T)$,
\begin{align}\label{lemma11}
\int_0^q \zeta g(\zeta, t)  d\zeta -\int_0^q \zeta g^{in}(\zeta )  d\zeta = -\int_{0}^{t} \int_0^q \int_{q- \zeta}^{\infty} \zeta \Psi(\zeta, \eta)g(\zeta, s) g(\eta, s) d\eta d\zeta ds.
\end{align}
\end{lem}

\begin{proof}
Set  $\omega (\zeta)=\zeta \chi_{(0, q)}(\zeta) $ {for $\zeta\in (0,\infty)$ and note that}
\begin{equation*}
\tilde{ \omega}(\zeta, \eta)=\begin{cases}
0,\ & \text{if}\ \zeta+\eta \in (0,q),\ \\
-(\zeta+\eta),\ &  \text{if}\ \zeta+\eta \geq q,\ (\zeta,\eta)\in (0,q)^2,\ \\
-\zeta, \ &  \text{if}\ (\zeta,\eta)\in (0,q)\times [q,\infty),\ \\
- \eta, \  &  \text{if}\ (\zeta,\eta) \in [q,\infty) \times (0,q),\ \\
0,\ &  \text{if}\ (\zeta,\eta)\in [q,\infty)^2.
\end{cases}
\end{equation*}
Inserting the above values of $\tilde{\omega}$ into \eqref{wsce} and using the symmetry of $\Psi$, we have
\begin{align*}
\int_0^{q}[ g({\zeta},t) - g^{in}({\zeta})] {\zeta} d{\zeta}=&\frac{1}{2}\int_0^t \int_0^{\infty} \int_{0}^{\infty}\tilde{\omega}({\zeta},{\eta}) \Psi({\zeta},{\eta})g({\zeta},s)g({\eta},s)d{\eta}d{\zeta}ds\nonumber\\
=&-\frac{1}{2}\int_0^t \int_0^{q} \int_{q-\zeta}^{q} (\zeta +\eta) \Psi({\zeta},{\eta})g({\zeta},s)g({\eta},s)d{\eta}d{\zeta}ds\nonumber\\
&-\frac{1}{2}\int_0^t \int_0^{q} \int_{q}^{\infty} \zeta \Psi({\zeta},{\eta})g({\zeta},s)g({\eta},s)d{\eta}d{\zeta}ds\nonumber\\
&-\frac{1}{2}\int_0^t \int_q^{\infty} \int_{0}^{q} \eta \Psi({\zeta},{\eta})g({\zeta},s)g({\eta},s)d{\eta}d{\zeta}ds\\
=&{-\int_0^t\int_0^q\int_{q-\zeta}^q \zeta \Psi({\zeta},{\eta})g({\zeta},s)g({\eta},s)d{\eta}d{\zeta}ds}\nonumber\\
&{-\int_0^t\int_0^q \int_q^\infty \zeta \Psi({\zeta},{\eta})g({\zeta},s)g({\eta},s)d{\eta}d{\zeta}ds},
\end{align*}
which completes the proof of Lemma~\ref{lemma1}.
\end{proof}

In order to complete the proof of Theorem~\ref{themcs}, it is sufficient to show that the right-hand side of (\ref{lemma11}) goes to zero as $q \to \infty$. The first step in that direction is the following result.

\begin{lem}\label{lemma2}
Assume that $(H1)$--{$(H2)$} hold. Let $g$ be a solution to (\ref{sce})--(\ref{1in1}) on $[0, T)$ and consider $t\in (0,T)$. Then
\begin{align*}
(i)\int_q^{\infty}[g(\zeta, t)-g^{in}(\zeta)]d\zeta=&- \frac{1}{2} \int_{0}^{t} \int_q^{\infty}\int_q^{\infty} \Psi(\zeta, \eta)g(\zeta, s) g(\eta, s)  d\eta d\zeta ds\\
&+ \frac{1}{2} \int_0^t \int_0^q \int_{q-\zeta}^{q} \Psi(\zeta, \eta)g(\zeta, s) g(\eta, s)  d\eta d\zeta ds,
\end{align*}

\begin{align*}
(ii)\lim_{q \to \infty} \int_{0}^{t} q \bigg[ \int_0^q \int_{q-\zeta}^{q} \Psi(\zeta, \eta)g(\zeta, s) g(\eta, s)  d\eta d\zeta- \int_q^{\infty}\int_q^{\infty} \Psi(\zeta, \eta)g(\zeta, s) g(\eta, s)  d\eta d\zeta \bigg]ds=0.
\end{align*}
\end{lem}

\begin{proof}
Set  $\omega(\zeta) =\chi_{[q, \infty)}(\zeta) $ for $\zeta\in (0,\infty)$ and the corresponding $\tilde{ \omega}$ is
\begin{equation*}
\tilde{ \omega}(\zeta, \eta)=\begin{cases}
0,\ & \text{if}\ \zeta+\eta \in (0,q),\\
1,\ &  \text{if}\ \zeta+\eta \in [q,\infty),\ (\zeta,\eta)\in (0,q)^2,\\
0, \ &  \text{if}\ (\zeta,\eta) \in (0,q)\times [q,\infty),\\
0, \  &  \text{if}\ (\zeta,\eta) \in [q,\infty) \times (0,q),\\
-1,\ &  \text{if}\ (\zeta,\eta) \in [q,\infty)^2.\\
\end{cases}
\end{equation*}

Inserting the above values of $\tilde{ \omega}$ into (\ref{wsce}), we obtain Lemma~\ref{lemma2} $(i)$.\\

Next, we readily infer from the integrability of $\zeta\mapsto \zeta g(\zeta,t)$ and $\zeta\mapsto \zeta g^{in}(\zeta)$ and Lebesgue's  dominated convergence theorem that
\begin{equation*}
\lim_{q\to\infty} q \int_q^\infty [g(\zeta, t)-g^{in}(\zeta)]d\zeta \le \lim_{q\to\infty} \int_q^\infty \zeta [g(\zeta, t)+ g^{in}(\zeta)]d\zeta = 0.
\end{equation*}
Multiplying the identity stated in Lemma~\ref{lemma2}~$(i)$ by $q$, we deduce from the previous statement that the left-hand side of the thus obtained identity converges to zero as $q\to\infty$. Then so does its right-hand side, which proves Lemma~\ref{lemma2}~$(ii)$.
\end{proof}

\begin{lem}\label{lemma4}
Assume that $(H1)$--$(H2)$ hold. Let $g$ be a weak solution to (\ref{sce})--(\ref{1in1}) on $[0,T)$. Then, for $t\in (0,T)$,
\begin{align*}
(i)\lim_{q \to \infty}\int_0^t \int_0^q \int_{q}^{\infty} \zeta \Psi({\zeta},{\eta})g({\zeta},s)g({\eta},s)d{\eta}d{\zeta} ds=0,
\end{align*}
and
\begin{align*}
(ii)\lim_{q \to \infty} q \int_0^t \int_q^{\infty} \int_{q}^{\infty} \Psi({\zeta},{\eta})g({\zeta},s)g({\eta},s)d{\eta}d{\zeta}  ds=0.
\end{align*}
\end{lem}

\begin{proof}
Let $q>1$, $t\in (0,T)$, and $s\in (0,t)$. To prove the first part of Lemma~\ref{lemma4}, we split the integral as follows
\begin{equation*}
\int_0^q \int_{q}^{\infty} \zeta \Psi({\zeta},{\eta}) g({\zeta},s) g({\eta},s) d{\eta} d{\zeta} = J_1(q,s) + J_2(q,s),
\end{equation*}
with
\begin{align*}
J_1(q,s) & := \int_0^1 \int_{q}^{\infty} \zeta \Psi({\zeta},{\eta}) g({\zeta},s) g({\eta},s) d{\eta} d{\zeta}, \\
J_2(q,s) & := \int_1^q \int_{q}^{\infty} \zeta \Psi({\zeta},{\eta}) g({\zeta},s) g({\eta},s) d{\eta} d{\zeta}.
\end{align*}
On the one hand, it follows from $(H2)$ and Young's inequality that
\begin{align*}
J_1(q,s) & \le k \int_0^1 \int_{q}^{\infty} \zeta^{1-\beta} \eta g({\zeta},s) g({\eta},s) d{\eta} d{\zeta} \\
& \le k \left( \int_0^\infty \zeta^{1-\beta} g(\zeta,s) d\zeta \right) \left( \int_q^\infty \eta g(\eta,s) d\eta \right) \\
& \le k \|g(s)\|_{L_{-2\beta,1}^1(0,\infty)} \int_q^\infty \eta g(\eta,s) d\eta
\end{align*}
and the integrability properties of $g$ from Definition~\ref{definition} and Lebesgue's dominated convergence theorem entail that
\begin{equation}
\lim_{q \to \infty}\int_0^t J_1(q,s) ds = 0. \label{PhLz1}
\end{equation}
On the other hand, we infer from (H2) that
\begin{align*}
J_2(q,s) & \le k \int_1^q \int_{q}^{\infty} \zeta (\zeta+\eta) g({\zeta},s) g({\eta},s) d{\eta} d{\zeta} \\
& \le 2k \int_1^q \int_{q}^{\infty} \zeta \eta g({\zeta},s) g({\eta},s) d{\eta} d{\zeta} \\
& \le 2k \mathcal{M}_1(g)(s) \int_q^\infty \eta g(\eta,s) d\eta,
\end{align*}
and we argue as above to conclude that
\begin{equation*}
\lim_{q \to \infty}\int_0^t J_2(q,s) ds = 0.
\end{equation*}
Recalling \eqref{PhLz1}, we have proved Lemma~\ref{lemma4}~$(i)$.\\

Similarly, by $(H2)$,
\begin{align*}
q \int_q^{\infty} \int_{q}^{\infty} \Psi({\zeta},{\eta}) g({\zeta},s) g({\eta},s) d{\eta} d{\zeta} & \le k \int_q^\infty \int_q^\infty (q\zeta+q\eta) g({\zeta},s) g({\eta},s) d{\eta} d{\zeta} \\
& \le 2k \int_q^\infty \int_{q}^{\infty} \zeta \eta g({\zeta},s) g({\eta},s) d{\eta} d{\zeta} \\
& \le 2k \mathcal{M}_1(g)(s)\int_q^\infty \eta g(\eta,s) d\eta,
\end{align*}
and we use once more the previous argument to obtain Lemma~\ref{lemma4}~$(ii)$.
\end{proof}

Now, we are in a position to prove Theorem~\ref{themcs}.

\begin{proof}[Proof of Theorem~\ref{themcs}] Let $t\in (0,T)$. From Lemma~\ref{lemma4}~$(i)$, we obtain
\begin{equation}\label{mce10}
\lim_{q \to \infty}\int_0^t \int_0^q \int_{q}^{\infty}\zeta \Psi({\zeta},{\eta})g({\zeta},s)g({\eta},s)d{\eta}d{\zeta} ds=0,
\end{equation}
while Lemma~\ref{lemma2}~$(ii)$ and Lemma~\ref{lemma4}~$(ii)$ imply that
\begin{equation}\label{mce12}
\lim_{q \to \infty} q \int_0^t \int_0^{q} \int_{q-\zeta}^{q} \Psi({\zeta},{\eta}) g({\zeta},s) g({\eta},s) d{\eta} d{\zeta} ds=0.
\end{equation}
Since
\begin{align*}
\int_0^t \int_0^{q} \int_{q-\zeta}^{\infty} \zeta \Psi({\zeta},{\eta}) g({\zeta},s) g({\eta},s) d{\eta} d{\zeta} ds & \le q \int_0^t \int_0^{q} \int_{q-\zeta}^{q} \Psi({\zeta},{\eta}) g({\zeta},s) g({\eta},s) d{\eta} d{\zeta} ds \\
& + \int_0^t \int_0^{q} \int_{q}^{\infty} \zeta \Psi({\zeta},{\eta}) g({\zeta},s) g({\eta},s) d{\eta} d{\zeta} ds,
\end{align*}
it readily follows from \eqref{mce10} and \eqref{mce12} that the right-hand side of \eqref{lemma11} converges to zero as $q\to\infty$. Consequently,
\begin{equation*}
\mathcal{M}_1(g)(t) = \lim_{q\to\infty} \int_0^q \zeta g(\zeta,s) d\zeta = \lim_{q\to\infty} \int_0^q \zeta g^{in}(\zeta) d\zeta = \mathcal{M}_1(g^{in}).
\end{equation*}
This completes the proof of Theorem~\ref{themcs}.
\end{proof}

\section{Existence of weak solutions}

This section is devoted to the construction of weak solutions to the SCE \eqref{sce}--\eqref{1in1} with a non-negative initial condition $g^{in}\in L_{-2\beta,1}^1(0,\infty)$. It is achieved by a classical compactness technique, the appropriate functional setting being here the space $L^1(0,\infty)$ endowed with its weak topology first used in the seminal work \cite{Stewart:1989} and subsequently further developed in \cite{Barik:2017, Camejo:2015I, Camejo:2015II, Escobedo:2003, Filbet:2004I, Giri:2012, Laurencot:2002L}.

Given a non-negative initial condition $g^{in}\in L_{-2\beta,1}^1(0,\infty)$, the starting point of this approach is the choice of an approximation of the SCE \eqref{sce}--\eqref{1in1}, which we set here to be
\begin{equation}\label{tncfe}
\frac{\partial g_n({\zeta},t)} {\partial t}  = \mathcal{B}_c(g_n)(\zeta,t) - \mathcal{D}_{c,n}^{\theta}(g_n)(\zeta,t),\ \qquad (\zeta,t)\in (0,n)\times (0, \infty),
\end{equation}
with truncated initial condition
\begin{equation}\label{nctp1a}
g_n(\zeta,0) = g_n^{in}({\zeta}):=g^{in}({\zeta})\chi_{(0,n)}({\zeta}), \qquad \zeta\in (0,n),
\end{equation}
where $n\ge 1$ is a positive integer, $\theta\in \{0,1\}$,
\begin{equation}
\Psi_n^\theta(\zeta,\eta) := \Psi(\zeta,\eta) \chi_{(1/n,n)}(\zeta) \chi_{(1/n,n)}(\eta) \left[ 1 - \theta + \theta \chi_{(0,n)}(\zeta+\eta) \right] \label{nctp1b}
\end{equation}
for $(\zeta,\eta)\in (0,\infty)^2$ and
\begin{equation}\label{Nonconservativedeath}
\mathcal{D}_{c,n}^{\theta}(g)(\zeta) := \int_{0}^{n-\theta\zeta} \Psi_n^\theta({\zeta},{\eta}) g({\zeta}) g ({\eta}) d{\eta},\ \qquad \zeta\in (0,n),\
\end{equation}
the gain term $\mathcal{B}_c(g)(\zeta)$ being still defined by \eqref{Birthterm} for $\zeta\in (0,n)$. The introduction of the additional parameter $\theta\in\{0,1\}$ allows us to handle simultaneously the so-called conservative approximation ($\theta=1$) and non-conservative approximation ($\theta=0$) and thereby prove that both approximations allow us to construct weak solutions to the SCE \eqref{sce}--\eqref{1in1}, a feature which is of interest when no general uniqueness result is available. Note that we also truncate the coagulation for small volumes to guarantee the boundedness of $\Psi_n^\theta$ which is a straightforward consequence of $(H2)$ and \eqref{nctp1b}. Thanks to this property, it follows from \cite{Stewart:1989} ($\theta=1$) and \cite{Filbet:2004I} ($\theta=0$) that there is a unique non-negative solution $g_n\in \mathcal{C}^1([0,\infty);L^1(0,n))$ to \eqref{tncfe}--\eqref{nctp1a} (we do not indicate the dependence upon $\theta$ for notational simplicity) which satisfies
\begin{equation}
\int_0^n{ {\zeta}g_n({\zeta},t)}d{\zeta} =\int_0^n{ {\zeta}g^{in}_n({\zeta})}d{\zeta} - (1-\theta) \int_0^t \int_0^n \int_{n-{\zeta}}^n {\zeta}\Psi_n^\theta(\zeta,\eta)g_n({\zeta},s)g_n({\eta},s)d\eta d{\zeta}ds \label{PhLz2}
\end{equation}
for $t\ge 0$. The second term in the right-hand side of \eqref{PhLz2} vanishes for $\theta=1$ and the total mass of $g_n$ remains constant throughout time evolution, which is the reason for this approximation to be called conservative. In contrast, when $\theta=0$, the total mass of $g_n$ decreases as a function of time. In both cases, it readily follows from \eqref{PhLz2} that
\begin{equation}\label{masstn}
\int_0^n{ {\zeta}g_n({\zeta},t)}d{\zeta} \leq \int_0^n{ {\zeta}g_n^{in}({\zeta})}d{\zeta} \le \mathcal{M}_1(g^{in}), \qquad t\ge 0.
\end{equation}

For further use, we next state the weak formulation of (\ref{tncfe})--(\ref{nctp1a}): for $t>0$ and $\omega \in L^{\infty}(0, n)$, there holds
\begin{equation}\label{nctp3}
\int_0^n \omega({\zeta})[g_n({\zeta},t)-g_n^{in}({\zeta})]d{\zeta}= \frac{1}{2} \int_0^t \int_{1/n}^n\int_{1/n}^{n} H_{\omega,n}^\theta({\zeta},{\eta})\Psi_n^{\theta}(\zeta,\eta) g_n({\zeta},s)g_n({\eta},s)d\eta d{\zeta}ds,
\end{equation}
where
\begin{equation*}
H_{\omega,n}^\theta({\zeta},{\eta}) := \omega ({\zeta}+{\eta})\chi_{(0,n)}({\zeta}+{\eta})- [\omega({\zeta})+\omega({\eta})] \left(1-\theta + \theta \chi_{(0,n)}(\zeta+\eta) \right)
\end{equation*}
for $(\zeta,\eta)\in (0,n)^2$.

In order to prove Theorem~\ref{thm1}, we shall show the convergence (with respect to an appropriate topology) of a subsequence of $(g_n)_{n\ge 1}$ towards a weak solution to \eqref{sce}--\eqref{1in1}. For that purpose, we now derive several estimates and first recall that, since $g^{in}\in L^1_{-2\beta, 1}(0, \infty)$, a refined version of de la Vall\'{e}e-Poussin theorem, see \cite{Le:1977} or \cite[Theorem~8]{Laurencot:2015}, guarantees that there exist two non-negative and convex functions $\sigma_1$ and $\sigma_2$ in $\mathcal{C}^2([0,\infty))$ such that $\sigma_1'$ and $\sigma_2'$ are concave,
\begin{equation}\label{convex1}
\sigma_i(0)=\sigma_i'(0)= 0,\ \ \lim_{x \to {\infty}}\frac{\sigma_i(x)}{x}=\infty,\ \ \ i=1,2,
\end{equation}
and
\begin{equation}\label{convex2}
\mathcal{I}_1 := \int_0^{\infty}\sigma_1({\zeta}) g^{in}({\zeta})d{\zeta}<\infty,\ \ \text{and}\ \ \mathcal{I}_2 := \int_0^{\infty}{\sigma_2\left( {\zeta}^{-\beta}g^{in}({\zeta}) \right)}d{\zeta}<\infty.
\end{equation}
Let us state the following properties of the above defined functions $\sigma_1$ and $\sigma_2$ which are required to prove Theorem~\ref{thm1}.

\begin{lem}\label{lemmaconvex} For $(x,y)\in (0,\infty)^2$, there holds
\begin{eqnarray*}
\hspace{-5cm}(i)~~~\sigma_2(x)\leq x\sigma_2'(x)\leq 2\sigma_2(x),\
\end{eqnarray*}
\begin{eqnarray*}
\hspace{-5cm}(ii)~~ x\sigma_2'(y)\leq \sigma_2(x)+  \sigma_2(y),\
\end{eqnarray*}
and
\begin{eqnarray*}
(iii)~0\leq \sigma_1(x+y)-\sigma_1(x)-\sigma_1(y)\leq 2\frac{x\sigma_1(y)+y\sigma_1(x)}{x+y}.
\end{eqnarray*}
\end{lem}

\begin{proof}
A proof of the statements $(i)$ and $(iii)$ may be found in \cite[Proposition~14]{Laurencot:2015} while $(ii)$ can easily be deduced from $(i)$ and the convexity of $\sigma_2$.
\end{proof}

We recall that throughout this section, the coagulation kernel $\Psi$ is assumed to satisfy $(H1)$--$(H2)$ and $g^{in}$ is a non-negative function in $L^1_{-2\beta, 1}(0,\infty )$.

\subsection{Moment estimates}

We begin with a uniform bound in $L_{-2\beta,1}^1(0,\infty)$.

\begin{lem}\label{LemmaEquibound}
There exists a positive constant $\mathcal{B}>0$ depending only on $g^{in}$ such that, for $t\ge 0$,
\begin{equation*}
\int_0^n \left( {\zeta}+{\zeta}^{-2\beta} \right) g_n(\zeta,t)d{\zeta} \leq \mathcal{B}.
\end{equation*}
\end{lem}

\begin{proof} Let $\delta\in (0,1)$ and take $\omega(\zeta)=(\zeta+\delta)^{-2\beta}$, $\zeta\in (0,n)$, in \eqref{nctp3}. With this choice of $\omega$,
\begin{equation*}
H_{\omega,n}^\theta(\zeta,\eta) \le \left[ (\zeta+\eta+\delta)^{-2\beta} - (\zeta+\delta)^{-2\beta} - (\eta+\delta)^{-2\beta} \right] \chi_{(0,n)}(\zeta+\eta) \le 0
\end{equation*}
for all $(\zeta,\eta)\in (0,n)^2$, so that \eqref{nctp3} entails that, for $t\ge 0$,
\begin{equation*}
\int_0^n (\zeta+\delta)^{-2\beta} g_n(\zeta,t) d\zeta \le \int_0^n (\zeta+\delta)^{-2\beta} g_n^{in}(\zeta) d\zeta \le \int_0^\infty \zeta^{-2\beta} g^{in}(\zeta) d\zeta.
\end{equation*}
We then let $\delta\to 0$ in the previous inequality and deduce from Fatou's lemma that
\begin{equation*}
\int_0^n \zeta^{-2\beta} g_n(\zeta,t) d\zeta \le \int_0^\infty \zeta^{-2\beta} g^{in}(\zeta) d\zeta,\ \qquad t\ge 0.
\end{equation*}
Combining the previous estimate with \eqref{masstn} gives Lemma~\ref{LemmaEquibound} with $\mathcal{B} := \|g^{in}\|_{L^1_{-2\beta, 1}(0,\infty)}$.
\end{proof}

We next turn to the control of the tail behavior of $g_n$ for large volumes, a step which is instrumental in the proof of the convergence of each integral on the right-hand side of \eqref{tncfe} to their respective limits on the right-hand side of \eqref{sce}.

\begin{lem}\label{massconservation}
For $T>0$, there is a positive constant $\Gamma(T)$ depending on $k$, $\sigma_1$, $g^{in}$, and $T$ such that,
\begin{eqnarray*}
(i)~~\sup_{t\in [0,T]}\int_0^n \sigma_1(\zeta) g_n(\zeta,t) d\zeta\leq \Gamma(T),
\end{eqnarray*}
and
\begin{eqnarray*}
(ii)~~ (1-\theta) \int_0^T \int_{1}^n \int_{1}^n \sigma_1(\zeta) \chi_{(0,n)}(\zeta+\eta) \Psi(\zeta,\eta) g_n({\zeta},s) g_n({\eta},s) d\eta d\zeta ds\leq \Gamma(T).
\end{eqnarray*}
\end{lem}

\begin{proof}
Let $T>0$ and $t\in (0,T)$. We set $\omega(\zeta)=\sigma_1(\zeta)$, $\zeta\in (0,n)$, into \eqref{nctp3} and obtain
\begin{align*}
& \int_0^n \sigma_1(\zeta)[g_n(\zeta,t)-g_n^{in}(\zeta)]d\zeta \\
& = \frac{1}{2} \int_0^t \int_{1/n}^n \int_{1/n}^n \tilde{\sigma}_1(\zeta,\eta) \chi_{(0,n)}(\zeta+\eta) \Psi(\zeta,\eta) g_n(\zeta,s) g_n(\eta,s) d\eta d\zeta ds \\
& \quad - \frac{1-\theta}{2} \int_0^t \int_{1/n}^n \int_{1/n}^n [\sigma_1(\zeta)+\sigma_1(\eta)] \chi_{[n,\infty)}(\zeta+\eta) \Psi(\zeta,\eta) g_n(\zeta,s) g_n(\eta,s) d\eta d\zeta ds,
\end{align*}
recalling that $\tilde{\sigma}_1(\zeta,\eta) = \sigma_1(\zeta+\eta)-\sigma_1(\zeta)-\sigma_1(\eta)$,  hence, using $(H2)$ and Lemma~\ref{lemmaconvex},
\begin{equation*}
\int_0^n \sigma_1(\zeta)[g_n(\zeta,t)-g_n^{in}(\zeta)]d\zeta \le \frac{k}{2} \sum_{i=1}^4 J_{i,n}(t) - (1-\theta) R_n(t), 
\end{equation*}
with
\begin{align*}
J_{1,n}(t) & := \int_0^t \int_0^1 \int_0^1 \tilde{\sigma}_1(\zeta,\eta) (\zeta\eta)^{-\beta} g_n(\zeta,s) g_n(\eta,s) d\eta d\zeta ds, \\
J_{2,n}(t) & := \int_0^t \int_0^1 \int_1^n \tilde{\sigma}_1(\zeta,\eta) \zeta^{-\beta} \eta g_n(\zeta,s) g_n(\eta,s) d\eta d\zeta ds, \\
J_{3,n}(t) & := \int_0^t \int_1^n \int_0^1 \tilde{\sigma}_1(\zeta,\eta) \zeta \eta^{-\beta} g_n(\zeta,s) g_n(\eta,s) d\eta d\zeta ds, \\
J_{4,n}(t) & := \int_0^t \int_1^n \int_1^n \tilde{\sigma}_1(\zeta,\eta) (\zeta+\eta) g_n(\zeta,s) g_n(\eta,s) d\eta d\zeta ds,
\end{align*}
and
\begin{equation*}
R_n(t) := \int_0^t \int_{1/n}^n \int_{1/n}^n \sigma_1(\zeta) \chi_{[n,\infty)}(\zeta+\eta) \Psi(\zeta,\eta) g_n(\zeta,s) g_n(\eta,s) d\eta d\zeta ds.
\end{equation*}
Owing to the concavity of $\sigma_1'$ and the property $\sigma_1(0)=0$, there holds
\begin{equation}
\tilde{\sigma}_1(\zeta,\eta) = \int_0^\zeta \int_0^\eta \sigma_1''(x+y) dydx \le \sigma_1''(0) \zeta \eta\ , \qquad (\zeta,\eta)\in (0,\infty)^2. \label{PhL2}
\end{equation}
By \eqref{PhL2}, Lemma~\ref{LemmaEquibound}, and Young's inequality,
\begin{align*}
J_{1,n}(t) & \le \sigma_1''(0) \int_0^t \int_0^1 \int_0^1 \zeta^{1-\beta} \eta^{-\beta} g_n(\zeta,s) g_n(\eta,s) d\eta d\zeta ds \\
& \le \sigma_1''(0) \int_0^t \left[ \int_0^1 \left( \zeta + \zeta^{-2\beta} \right) g_n(\zeta,s) d\zeta \right]^2 ds \le \sigma_1''(0) \mathcal{B}^2 t.
\end{align*}
Next, Lemma~\ref{lemmaconvex}~$(iii)$, Lemma~\ref{LemmaEquibound}, and Young's inequality give
\begin{align*}
J_{2,n}(t) = J_{3,n}(t) & \le 2 \int_0^t \int_0^1 \int_1^n \frac{\zeta \sigma_1(\eta) + \eta \sigma_1(\zeta)}{\zeta+\eta} \zeta^{-\beta} \eta g_n(\zeta,s) g_n(\eta,s) d\eta d\zeta ds \\
& \le 2 \int_0^t \int_0^1 \int_1^n \left[ \zeta^{1-\beta} \sigma_1(\eta) + \sigma_1(1) \zeta^{-\beta} \eta \right] g_n(\zeta,s) g_n(\eta,s) d\eta d\zeta ds \\
& \le 2 \int_0^t \left[ \int_0^1 \left( \zeta + \zeta^{-2\beta} \right) g_n(\zeta,s) d\zeta \right] \left[ \int_1^n \sigma_1(\eta) g_n(\eta,s) d\eta \right] ds \\
& \quad + \sigma_1(1) \int_0^t \left[ \int_0^1 \left( \zeta + \zeta^{-2\beta} \right) g_n(\zeta,s) d\zeta \right] \left[ \int_1^n \eta g_n(\eta,s) d\eta \right] ds \\
& \le 2 \sigma_1(1) \mathcal{B}^2 t + 2 \mathcal{B} \int_0^t \int_0^n \sigma_1(\eta) g_n(\eta,s) d\eta ds,
\end{align*}
and
\begin{align*}
J_{4,n}(t) & \le 2 \int_0^t \int_1^n \int_1^n \left( \eta \sigma_1(\zeta) + \zeta \sigma_1(\eta) \right) g_n(\zeta,s) g_n(\eta,s) d\eta d\zeta ds \\
& \le 4 \mathcal{B} \int_0^t \int_0^n \sigma_1(\eta) g_n(\eta,s) d\eta ds.
\end{align*}
Gathering the previous estimates, we end up with
\begin{align*}
\int_0^n \sigma_1(\zeta)[g_n(\zeta,t)-g_n^{in}(\zeta)]d\zeta & \le k \left( \frac{\sigma_1''(0)}{2} + 2 \sigma_1(1) \right) \mathcal{B}^2 t \\
& \quad + 4k\mathcal{B} \int_0^t \int_0^n \sigma_1(\eta) g_n(\eta,s) d\eta ds - (1-\theta) R_n(t),
\end{align*}
and we infer from Gronwall's lemma and \eqref{convex2} that
\begin{align*}
\int_0^n \sigma_1(\zeta) g_n(\zeta,t) d\zeta + (1-\theta) R_n(t) & \le e^{4k\mathcal{B}t} \int_0^n \sigma_1(\zeta) g_n^{in}(\zeta) d\zeta + \left( \frac{\sigma_1''(0)}{8} + \frac{\sigma_1(1)}{2} \right) \mathcal{B} e^{4k\mathcal{B}t} \\
& \le \left[ \mathcal{I_1} + \left( \sigma_1''(0) + \sigma_1(1) \right) \mathcal{B} \right] e^{4k\mathcal{B}t}.
\end{align*}
This completes the proof of Lemma~\ref{massconservation}.
\end{proof}

\subsection{Uniform integrability}

Next, our aim being to apply Dunford-Pettis' theorem, we have to prevent concentration of the sequence $(g_n)_{n\ge 1}$ on sets of arbitrary small measure. For that purpose, we need to show the following result.

\begin{lem}\label{LemmaEquiintegrable}
For any $T>0$ and $\lambda >0$, there is a positive constant $\mathcal{L}_1(\lambda, T)$ depending only on $k$, $\sigma_2$, $g^{in}$, ${\lambda}$,  and $T$ such that
\begin{equation*}
\sup_{t\in [0,T]}\int_0^{\lambda} \sigma_2\left( {\zeta}^{-\beta}  g_n(\zeta,t) \right) d\zeta \leq \mathcal{L}_1({\lambda},T).
\end{equation*}
\end{lem}

\begin{proof} {For $(\zeta,t)\in (0,n)\times (0,\infty)$, we set $u_n({\zeta},t):={\zeta}^{-\beta} g_n(\zeta,t)$. Let $\lambda\in (1,n)$, $T>0$, and $t\in (0,T)$. Using Leibniz's rule, Fubini's theorem, and (\ref{tncfe}), we obtain
\begin{align}\label{Equiint1}
\hspace{-.2cm}\frac{d}{dt}\int_0^{\lambda} \sigma_2(u_n({\zeta},t)) d\zeta \leq & \frac{1}{2}\int_0^{\lambda}\hspace{-.1cm}\int_0^{\lambda -\eta} \sigma_2{'}(u_n({\zeta+\eta},t)) {(\zeta+\eta)}^{-\beta} \Psi_n^\theta({\zeta},{\eta}) g_n({\zeta},t) g_n(\eta,t) d{\zeta} d{\eta}.
\end{align}
It also follows from $(H2)$ that
\begin{equation}
\Psi_n^\theta(\zeta,\eta) \le \Psi(\zeta,\eta) \le 2k \lambda^{1+2\beta} (\zeta \eta)^{-\beta}, \qquad (\zeta,\eta)\in (0,\lambda)^2. \label{PhL3}
\end{equation}
We then infer from \eqref{Equiint1}, \eqref{PhL3}, Lemma~\ref{lemmaconvex}~$(ii)$ and Lemma~\ref{LemmaEquibound} that
\begin{align*}
\frac{d}{dt}\int_0^{\lambda} \sigma_2(u_n({\zeta},t))d{\zeta} \leq & k \lambda^{1+2\beta} \int_0^{\lambda} \int_0^{{\lambda}-{\eta}}\sigma_2^{'}(u_n({\zeta}+{\eta},t)) ({\zeta}+{\eta})^{-\beta} u_n(\zeta,t) u_n(\eta,t)d{\zeta} d{\eta}\nonumber\\
\leq & k \lambda^{1+2\beta} \int_0^{\lambda} \int_0^{{\lambda}-{\eta}} {\eta}^{-\beta} \left[ \sigma_2(u_n(\zeta+\eta,t)) + \sigma_2(u_n(\zeta,t)) \right] u_n(\eta,t) d\zeta d\eta\nonumber\\
\leq & 2 k \lambda^{1+2\beta} \int_0^{\lambda} \eta^{-2\beta} g_n(\eta,t) \int_0^{{\lambda}-{\eta}} \sigma_2(u_n({\zeta}+{\eta},t)) d\zeta d\eta\nonumber\\
\leq & 2 k \lambda^{1+2\beta} \mathcal{B}  \int_0^\lambda \sigma_2(u_n(\zeta,t)) d\zeta.
\end{align*}}
Then, using Gronwall's lemma, the monotonicity of $\sigma_2$, and \eqref{convex2}, we obtain
\begin{eqnarray*}
\int_0^{\lambda} \sigma_2({\zeta}^{-\beta} g_n(\zeta,t)) d\zeta \leq \mathcal{L}_1({\lambda}, T),
\end{eqnarray*}
where $\mathcal{L}_1({\lambda, T}):=\mathcal{I}_2 e^{2 k \lambda^{1+2\beta} \mathcal{B} T}$, and the proof is complete.
\end{proof}

\subsection{Time equicontinuity}

The outcome of the previous sections settles the (weak) compactness issue with respect to the volume variable. We now turn to the time variable.

\begin{lem}\label{timequic}
Let $t_2\ge t_1 \ge 0$ and $\lambda\in (1,n)$. There is a positive constant $\mathcal{L}_2(\lambda)$ depending only on $k$, $g^{in}$, and $\lambda$ such that
\begin{equation*}
\int_0^\lambda \zeta^{-\beta} |g_n(\zeta,t_2) - g_n(\zeta,t_1)| d\zeta \le \mathcal{L}_2(\lambda) (t_2-t_1).
\end{equation*}
\end{lem}

\begin{proof}
Let $t>0$. On the one hand, by Fubini's theorem, \eqref{PhL3}, and Lemma~\ref{LemmaEquibound},
\begin{align*}
\int_0^\lambda \zeta^{-\beta} \mathcal{B}_c(g_n)(\zeta,t) d\zeta & \le \frac{1}{2} \int_0^\lambda \int_0^{\lambda-\zeta} (\zeta+\eta)^{-\beta} \Psi(\zeta,\eta) g_n(\zeta,t) g_n(\eta,t) d\eta d\zeta \\
& \le k \lambda^{1+2\beta} \int_0^\lambda \int_0^\lambda \zeta^{-\beta} \eta^{-2\beta} g_n(\zeta,t) g_n(\eta,t) d\eta d\zeta \\
& \le k \lambda^{1+3\beta} \left( \int_0^\lambda \zeta^{-2\beta} g_n(\zeta,t) d\zeta \right)^2 \le k \lambda^{1+3\beta} \mathcal{B}^2.
\end{align*}
On the other hand, since
\begin{equation*}
\Psi_n^\theta(\zeta,\eta) \le \Psi(\zeta,\eta) \le 2k \lambda^\beta \eta \zeta^{-\beta},\ \qquad 0 < \zeta < \lambda < \eta < n,
\end{equation*}
we infer from \eqref{PhL3} and Lemma~\ref{LemmaEquibound} that
\begin{align*}
\int_0^\lambda \zeta^{-\beta} \mathcal{D}_{c,n}^\theta(g_n)(\zeta,t) d\zeta & \le \int_0^\lambda \int_0^n \zeta^{-\beta} \Psi(\zeta,\eta) g_n(\zeta,t) g_n(\eta,t) d\eta d\zeta \\
& \le 2k\lambda^{1+2\beta} \int_0^\lambda \int_0^\lambda \zeta^{-2\beta} \eta^{-\beta} g_n(\zeta,t) g_n(\eta,t) d\eta d\zeta \\
& \quad + 2k\lambda^\beta \int_0^\lambda \int_\lambda^n \zeta^{-\beta} \eta g_n(\zeta,t) g_n(\eta,t) d\eta d\zeta \\
& \le 2k \mathcal{B}^2 (1+\lambda^{1+\beta}) \lambda^{\beta}.
\end{align*}
Consequently, by \eqref{tncfe},
\begin{align*}
\int_0^\lambda \zeta^{-\beta} |g_n(\zeta,t_2) - g_n(\zeta,t_1)| d\zeta & \le \int_{t_1}^{t_2} \int_0^\lambda \zeta^{-\beta} \left| \frac{\partial g_n}{\partial t}(\zeta,t) \right| d\zeta dt \\
& \le \int_{t_1}^{t_2} \int_0^\lambda \zeta^{-\beta} \left[ \mathcal{B}_c(g_n)(\zeta,t) + \mathcal{D}_{c,n}^\theta(g_n)(\zeta,t) \right] d\zeta \\
& \le k \mathcal{B}^2 (2+2\lambda^{1+\beta}+\lambda^{1+2\beta}) \lambda^{\beta} (t_2-t_1),
\end{align*}
which completes the proof with $\mathcal{L}_2(\lambda):=k \mathcal{B}^2 (2+2\lambda^{1+\beta}+\lambda^{1+2\beta}) \lambda^{\beta}$.
\end{proof}

\subsection{Convergence}

We are now in a position to complete the proof of the existence of a weak solution to the SCE \eqref{sce}--\eqref{1in1}.

\begin{proof}[Proof of Theorem~\ref{thm1}] For $(\zeta,t)\in (0,n)\times (0,\infty)$, we set $u_n({\zeta},t):={\zeta}^{-\beta} g_n(\zeta,t)$. Let $T>0$ and $\lambda>1$. Owing to the superlinear growth \eqref{convex1} of $\sigma_2$ at infinity and Lemma~\ref{LemmaEquiintegrable}, we infer from Dunford-Pettis' theorem that there is a weakly compact subset $\mathcal{K}_{\lambda,T}$ of $L^1(0,\lambda)$ such that $(u_n(t))_{n\ge 1}$ lies in $\mathcal{K}_{\lambda,T}$ for all $t\in [0,T]$. Moreover, by Lemma~\ref{timequic}, $(u_n)_{n\ge 1}$ is strongly equicontinuous in $L^1(0,\lambda)$ at all $t\in (0,T)$ and thus also weakly equicontinuous in $L^1(0,\lambda)$ at all $t\in (0,T)$. A variant of \emph{Arzel\`{a}-Ascoli's theorem} \cite[Theorem~1.3.2]{Vrabie:1995} then guarantees that $(u_n)_{n\ge 1}$ is relatively compact in $\mathcal{C}_w([0,T];L^1(0,\lambda))$. This property being valid for all $T>0$ and $\lambda>1$, we use a diagonal process to obtain a subsequence of $(g_n)_{n\ge 1}$ (not relabeled) and a non-negative function $g$ such that
\begin{equation*}
g_n \longrightarrow g ~~~\mbox{in} ~~\mathcal{C}_w([0,T]; L^1(0,\lambda))
\end{equation*}
for all $T>0$ and $\lambda>1$. Owing to Lemma~\ref{massconservation} and the superlinear growth \eqref{convex1} of $\sigma_1$ at infinity, a by-now classical argument allows us to improve the previous convergence to
\begin{equation}
g_n \longrightarrow g ~~~\mbox{in} ~~\mathcal{C}_w([0,T]; L^1((0,\infty);(\zeta^{-\beta}+\zeta)d\zeta)). \label{PhL4}
\end{equation}

To complete the proof of Theorem~\ref{thm1}, it remains to show that $g$ is a weak solution to the SCE \eqref{sce}--\eqref{1in1} on $[0,\infty)$ in the sense of Definition~\ref{definition}. This step is carried out by the classical approach of \cite{Stewart:1989} with some modifications as in \cite{Camejo:2015I, Camejo:2015II} and \cite{Laurencot:2002L} to handle the convergence of the integrals for small and large volumes, respectively. In particular, on the one hand, the behavior for large volumes is controlled by the estimates of Lemma~\ref{massconservation} with the help of the superlinear growth \eqref{convex1} of $\sigma_1$ at infinity and the linear growth $(H2)$ of $\Psi$. On the other hand, the behavior for small volumes is handled by $(H2)$, Lemma~\ref{LemmaEquibound}, and \eqref{PhL4}.\\

Finally, $g$ being a weak solution to \eqref{sce}--\eqref{1in1} on $[0,\infty)$ in the sense of Definition~\ref{definition}, it is mass-conserving according to Theorem~\ref{themcs}, which completes the proof of Theorem~\ref{thm1}.
\end{proof}

\section*{Acknowledgments}

This work was supported by University Grant Commission (UGC), India, for providing Ph.D fellowship to PKB. AKG would like to thank  Science and Engineering Research Board (SERB), Department of Science and Technology (DST), India for providing funding support through the project $YSS/2015/001306$.

\bibliographystyle{plain}

\end{document}